\documentclass[12pt,reqno,openbib,runningheads,a4paper]{amsart}%
\usepackage{graphicx}
\usepackage{amssymb}
\usepackage{amsfonts}
\usepackage{amsmath}
\usepackage{graphicx}
\usepackage{lineno}
\usepackage[authoryear]{natbib}
\usepackage[legalpaper,bookmarks=true,colorlinks=true,linkcolor=blue,citecolor=blue]%
{hyperref}
\usepackage{color}%
\providecommand{\U}[1]{\protect\rule{.1in}{.1in}}
\providecommand{\U}[1]{\protect\rule{.1in}{.1in}}
\newtheorem{theorem}{Theorem}[section]

\newtheorem{proposition}{Proposition}[section]

\newtheorem{lemma}{Lemma}[section]

\makeatletter
\renewcommand{\@biblabel}[1]{}
\makeatother
\begin{document}

\begin{center}
{\Large \textbf{Statistical estimate of the proportional hazard premium of
loss under random censoring}}\medskip\medskip

{\large Louiza Soltane, Djamel Meraghni, Abdelhakim Necir}$^{\ast}$\medskip

{\small \textit{Laboratory of Applied Mathematics, Mohamed Khider University,
Biskra, Algeria}}\bigskip\medskip
\end{center}

\noindent\textbf{Abstract}\smallskip

\noindent Many insurance premium principles are defined and various estimation
procedures introduced in the literature. In this paper, we focus on the
estimation of the excess-of-loss reinsurance premium when the risks are
randomly right-censored. The asymptotic normality of the proposed estimator is
established under suitable conditions and its performance evaluated through
sets of simulated data.\medskip

\medskip


\vfill

\noindent{\small $^{\text{*}}$Corresponding author:
\texttt{necirabdelhakim@yahoo.fr} \newline\noindent\textit{E-mail
addresses:}\newline\texttt{louiza\_stat@yahoo.com} (L.~Soltane)\newline%
\texttt{djmeraghni@yahoo.com} (D.~Meraghni)}

\section{\textbf{Introduction\label{sec1}}}

\noindent Let $X_{1},...,X_{n}$ be $n\geq1$ independent copies of a
non-negative random variable (rv) $X,$ defined over some probability space
$(\Omega,\mathcal{A},\mathbb{P)},$ with continuous cumulative distribution
function (cdf) $F.$ An independent sequence of independent rv's $Y_{1}%
,...,Y_{n}$ with continuous cdf $G$ censor them to the right, so that at each
stage $j$ we only can observe $Z_{j}:=\min(X_{j},Y_{j})$ and the variable
$\delta_{j}:=\mathbf{1}\left\{  X_{j}\leq Y_{j}\right\}  $ (with
$\mathbf{1}\left\{  \cdot\right\}  $ denoting the indicator function)
informing whether or not there has been censorship. This model is very useful
in a variety of areas where random censoring is very likely to occur such as
in biostatistics, medical research, reliability analysis, actuarial
science,... For more general censoring schemes and other issues involving
censored data, we refer, for instance, to \cite{CO84}, \cite{KP80} and
\cite{Gill80}.\smallskip

\noindent In insurance, the worst scenarios are those caused by extreme events
such as natural catastrophes, human-made disasters and financial crashes.
These events increase the bill of insurance and reinsurance companies. A
typical requirement for actuaries is the determination of adequate premiums
for such risks. Usually, the insurer's claims data do not correspond to the
underlying losses, because they are censored from above, since the insurer
stipulates an upper limit to the amount to be paid out and the reinsurer
covers the excess over this fixed threshold. This kind of reinsurance is
called excess-of-loss reinsurance (see, e.g., \citeauthor{RSHSVT99},
\citeyear{RSHSVT99}; \citeauthor{EKM97}, \citeyear{EKM97}) and the upper limit
has distinct designations that are specific to each insurance type. For
instance, in life insurance, it is called the cedent's company retention level
while in non-life insurance, it is called the deductible, where the losses
should be treated separately. For a discussion on the occurrence of
right-random censorship in the area of insurance, one refers to \cite{DPVK06}
in which a study on the allocated loss adjustment expenses (ALAE's) is
given.\smallskip

\noindent Let us assume that both $F$ and $G$ are heavy-tailed, that is there
exist two constants $\gamma_{1}>0$ and $\gamma_{2}>0,$ called tail indices or
extreme value indices (EVI's), such that%
\begin{equation}
\overline{F}(z)\sim z^{-1/\gamma_{1}}\ell_{1}(z)\text{ and }\overline
{G}(z)\sim z^{-1/\gamma_{2}}\ell_{2}(z),\text{ as }z\rightarrow\infty,
\label{VR}%
\end{equation}
where $\ell_{1}$ and $\ell_{2}$ are slowly varying functions at infinity, i.e.
$\lim_{z\rightarrow\infty}\ell_{i}(xz)/\ell_{i}(z)=1$ for every $x>0,$
$i=1,2.$ Throughout the paper, we use the notation $\overline{\mathcal{S}%
}(x):=\mathcal{S}(\infty)-\mathcal{S}(x),$ for any function $\mathcal{S}%
\left(  x\right)  $ of $x>0.$ If relations $\left(  \ref{VR}\right)  $ hold,
then we have, for any $x>0$%
\begin{equation}
\underset{z\rightarrow\infty}{\lim}\frac{\overline{F}(xz)}{\overline{F}%
(z)}=x^{-1/\gamma_{1}}\text{ and }\underset{z\rightarrow\infty}{\lim}%
\frac{\overline{G}(xz)}{\overline{G}(z)}=x^{-1/\gamma_{2}}, \label{Condition}%
\end{equation}
and we say that $\overline{F}$ and $\overline{G}$ are regularly varying at
infinity as well, with respective tail indices $-1/\gamma_{1}$ and
$-1/\gamma_{2}$, which we denote by $\overline{F}\in\mathcal{RV}%
_{-1/\gamma_{1}}$ and $\overline{G}\in\mathcal{RV}_{-1/\gamma_{2}}$. Note
that, in virtue of the independence of $X$ and $Y,$ the cdf of the observed
$Z^{\prime}s,$ that we denote by $H,$ is also heavy-tailed and we have
$H\in\mathcal{RV}_{-1/\gamma}$ with $\gamma:=\gamma_{1}\gamma_{2}/(\gamma
_{1}+\gamma_{2}).$ This class of distributions, which includes models such as
Pareto, Burr, Fr\'{e}chet, L\'{e}vy-stable and log-gamma, plays a prominent
role in extreme value theory. Also known as Pareto-type or Pareto-like
distributions, these models have important practical applications and are used
rather systematically in certain branches of non-life insurance as well as in
finance, telecommunications, geology and many other fields (see e.g.
\citeauthor{R07}, \citeyear{R07}). The analysis of extreme values of randomly
censored data is a new research topic to which \cite{RT97} made a very brief
reference, in Section 6.1, as a first step but with no asymptotic results. In
the last decade, several authors started to be interested in the estimation of
the tail index along with large quantiles under random censoring as one can
see in \cite{GO03}, \cite{BGDF07}, \cite{EFG08} and \cite{WW14}. \cite{GN11}
also made a contribution to this field by providing a detailed simulation
study and applying the estimation procedures on some survival data sets. Let
$\left\{  \left(  Z_{i},\delta_{i}\right)  ,\text{ }1\leq i\leq n\right\}  $
be a sample from the couple of rv's $\left(  Z,\delta\right)  $ and
$Z_{1:n}\leq...\leq Z_{n:n}$\ the order statistics pertaining to $\left(
Z_{1},...,Z_{n}\right)  .$ If we denote the concomitant of the $i$th order
statistic by $\delta_{\left[  i:n\right]  }$ (i.e. $\delta_{\left[
i:n\right]  }=\delta_{j}$ if $Z_{i:n}=Z_{j}),$\ then Hill's estimator of
$\gamma_{1}$ adapted to censored data is defined as $\widehat{\gamma}%
_{1}^{(H,c)}:=\widehat{\gamma}^{H}/\widehat{p},$ where $\widehat{\gamma}%
^{H}:=k^{-1}\sum_{i=1}^{k}\log\left(  Z_{n-i+1:n}/Z_{n-k:n}\right)  $
represents Hill's estimator (\citeauthor{Hill75}, \citeyear{Hill75}) of
$\gamma,$ with $k=k_{n}$ being an integer sequence satisfying%
\begin{equation}
1<k<n,\text{ }k\rightarrow\infty\text{ and }k/n\rightarrow0\text{ as
}n\rightarrow\infty, \label{k}%
\end{equation}
and $\widehat{p}:=k^{-1}\sum_{i=1}^{k}\delta_{\left[  n-i+1:n\right]  }$ being
the proportion of upper non-censored observations. \cite{EFG08} established
the asymptotic normality of\ $\widehat{\gamma}_{1}^{(H,c)}$ by assuming that
cdf's are absolutely continuous. Recently, \cite{BMN15} proved that
$\widehat{p}$ consistently estimates $p:=\gamma_{2}/(\gamma_{1}+\gamma_{2})$
leading to the consistency of $\widehat{\gamma}_{1}^{(H,c)}.$ They also
established the asymptotic normality of $\widehat{\gamma}_{1}^{(H,c)}$ by
adopting an approach that is different from that of \cite{EFG08}.\smallskip

\noindent In the excess-of-loss reinsurance treaty, the ceding company covers
claims which do not exceed a (high) number $R\geq0$ (called retention level),
while the reinsurer pays the part $(X_{i}-R)_{+}:=\max\left(  0,X_{i}%
-R\right)  $ of each claim beyond $R.$ Applying Wang's premium calculation
principle (\citeauthor{Wang96}, \citeyear{Wang96}), with a distortion function
equal to\ $x^{1/\rho},$ one defines what is called the proportional hazard
premium (PHP), where $\rho\geq1$ represents the distortion parameter or the
risk aversion index.\ Then, the PHP of loss for the layer from $R$ to infinity
is defined\ as follows:%
\[
\Pi_{\rho}(R):=\int_{R}^{\infty}\left(  \overline{F}(x)\right)  ^{1/\rho}dx,
\]
which may be rewritten into
\[
\Pi_{\rho}(R)=R(\overline{F}(R))^{1/\rho}%
{\displaystyle\int_{1}^{\infty}}
\left(  \frac{\overline{F}\left(  Rx\right)  }{\overline{F}\left(  R\right)
}\right)  ^{1/\rho}dx.
\]
By using the well-known Karamata theorem (see, for instance,
\citeauthor{deHF06}, \citeyear[page 363]{deHF06}), we get%
\[
\Pi_{\rho}(R)\sim\frac{\rho}{1/\gamma_{1}-\rho}R\left(  \overline{F}\left(
R\right)  \right)  ^{1/\rho},\text{ }0<\gamma_{1}<1/\rho,
\]
for large $R.$ Since $\overline{F}\in\mathcal{RV}_{-1/\gamma_{1}},$ then
$\overline{F}\left(  x\right)  \sim\overline{F}\left(  h\right)  \left(
x/h\right)  ^{-1/\gamma_{1}}$ as $x\rightarrow\infty,$ where $h=h_{n}%
:=H^{\leftarrow}\left(  1-k/n\right)  $ with $H^{\leftarrow}\left(  y\right)
:=\inf\left\{  x:H\left(  x\right)  \geq y\right\}  ,$ $0<y<1,$ denoting the
quantile function pertaining to $H.$ This leads us to derive a Weissman-type
estimator (see \citeauthor{Weissman78}, \citeyear{Weissman78}) for the
distribution tail $\overline{F}$ for censored data as follows:%
\[
\widehat{\overline{F}}\left(  x\right)  =\left(  \frac{x}{Z_{n-k:n}}\right)
^{-1/\widehat{\gamma}_{1}^{(H,c)}}\overline{F}_{n}\left(  Z_{n-k:n}\right)  .
\]
In the context of randomly right censored observations, the nonparametric
maximum likelihood estimator of $F$ is given by \cite{KM58} as the product
limit estimator%
\[
\overline{F}_{n}(x):=%
{\displaystyle\prod\limits_{Z_{i:n}\leq x}}
\left(  1-\dfrac{\delta_{\left[  i:n\right]  }}{n-i+1}\right)  =%
{\displaystyle\prod\limits_{Z_{i:n}\leq x}}
\left(  \dfrac{n-i}{n-i+1}\right)  ^{\delta_{\left[  i:n\right]  }},\text{ for
}x<Z_{n:n},
\]
which gives $\overline{F}_{n}(Z_{n-k:n})=%
{\displaystyle\prod\nolimits_{i=1}^{n-k}}
\left(  1-\frac{\delta_{\left[  i:n\right]  }}{n-i+1}\right)  .$ Thus, the
distribution tail estimator is of the form%
\[
\widehat{\overline{F}}\left(  x\right)  :=\left(  \frac{x}{Z_{n-k:n}}\right)
^{-1/\widehat{\gamma}_{1}^{(H,c)}}%
{\displaystyle\prod\limits_{i=1}^{n-k}}
\left(  1-\frac{\delta_{\left[  i:n\right]  }}{n-i+1}\right)  ,
\]
and consequently, we define the PHP estimator as follows:%

\[
\widehat{\Pi}_{\rho}(R):=\frac{\rho R}{1/\widehat{\gamma}_{1}^{(H,c)}-\rho
}\left(  \frac{R}{Z_{n-k:n}}\right)  ^{-1/\left(  \rho\widehat{\gamma}%
_{1}^{(H,c)}\right)  }%
{\displaystyle\prod\limits_{i=1}^{n-k}}
\left(  1-\frac{\delta_{\left[  i:n\right]  }}{n-i+1}\right)  ^{1/\rho}.
\]
The outline of the paper is as follows. In Section \ref{sec2}, we state our
main result that consists in the asymptotic normality of the newly proposed
estimator $\widehat{\Pi}_{\rho}(R),$ which we prove in Section \ref{sec4}. In
Section \ref{sec3}, we carry out a simulation study to illustrate its finite
sample behavior. Finally, some results, that are instrumental to our needs,
are gathered in the Appendix.

\section{Main results\textbf{\label{sec2}}}

\noindent It is well-known that the asymptotic normality of extreme value
theory based estimators is adequately achieved within the second-order
framework (see \citeauthor{deHS96}, \citeyear{deHS96}). Thus, it seems quite
natural to suppose that cdf's $F$ and $G$ satisfy the well-known second-order
condition of regular variation. That is, we assume that there exist two
constants $\tau_{j}\leq0$ (called second-order parameters) and two functions
$A_{j},$ $j=1,2,$ tending to zero and not changing sign near infinity, such
that for any $x>0$%
\begin{equation}%
\begin{array}
[c]{l}%
\underset{t\rightarrow\infty}{\lim}\dfrac{\overline{F}(tx)/\overline
{F}(t)-x^{-1/\gamma_{1}}}{A_{1}(t)}=x^{-1/\gamma_{1}}\dfrac{x^{\tau_{1}%
/\gamma_{1}}-1}{\gamma_{1}\tau_{1}},\medskip\\
\underset{t\rightarrow\infty}{\lim}\dfrac{\overline{G}(tx)/\overline
{G}(t)-x^{-1/\gamma_{2}}}{A_{2}(t)}=x^{-1/\gamma_{2}}\dfrac{x^{\tau_{2}%
/\gamma_{2}}-1}{\gamma_{2}\tau_{2}}.
\end{array}
\label{second-order}%
\end{equation}

\begin{theorem}
\label{Theo}Assume that the second-order conditions of regular variation
$\left(  \ref{second-order}\right)  $ hold, with $0<\gamma_{1}<1/\rho$ and let
$k=k_{n}$ be an integer sequence satisfying, in addition to $\left(
\ref{k}\right)  ,$ $\sqrt{k}A_{1}(h)\rightarrow\lambda_{1}.$ Assume further
that $R/h\rightarrow1.$ Then%
\[
\sqrt{k}\dfrac{\widehat{\Pi}_{\rho}(R)-\Pi_{\rho}(R)}{\left(  R/h\right)
^{-1/\rho\gamma_{1}}R\left(  \overline{F}\left(  h\right)  \right)  ^{1/\rho}%
}\overset{\mathcal{D}}{\rightarrow}\mathcal{N}\left(  \mu,\sigma^{2}\right)
,\text{ as }n\rightarrow\infty,
\]
where%
\[
\mu:=\dfrac{\rho\lambda_{1}}{\left(  1-p\tau_{1}\right)  \left(  1-\rho
\gamma_{1}\right)  ^{2}}+\dfrac{\lambda_{1}}{\rho\left(  \gamma_{1}+\tau
_{1}+\rho-2\right)  \left(  2-\rho-\gamma_{1}\right)  },
\]
and%
\[
\sigma^{2}:=\frac{\gamma_{1}^{2}}{\left(  1-\rho\gamma_{1}\right)  ^{2}%
}\left(  p\left(  2-p\right)  +\frac{\rho\left(  p-1\right)  }{\left(
1-\rho\gamma_{1}\right)  }+\frac{\rho^{2}\left(  1-2p\right)  }{p\left(
1-\rho\gamma_{1}\right)  ^{2}}\right)  .
\]

\end{theorem}

\section{Simulation study\textbf{\label{sec3}}}

\noindent We carry out a simulation study to illustrate the performance of our
estimator, through two sets of censored and censoring data, both drawn from
the following Burr model. That is%
\[
\overline{F}\left(  x\right)  =\left(  1+x^{\eta/\gamma_{1}}\right)
^{-1/\eta}\text{ and }\overline{G}\left(  x\right)  =\left(  1+x^{\eta
/\gamma_{2}}\right)  ^{-1/\eta},\text{ }x\geq0,
\]
where $\gamma_{1},\gamma_{2}>0.$ We fix $\eta=1/4,$ we choose the values
$0.10$ and $0.25$ for $\gamma_{1}$ and two distinct aversion index values
$\rho=1.00$ and $\rho=1.10.$ For the proportion of the really observed extreme
values, we take $p=0.40,$ $0.60$ and $0.80,$ that is, we allow the percentage
of censoring in the right tail of $X$ to be $60\%,$ $40\%$ and $20\%$. For
each couple $\left(  \gamma_{1},p\right)  ,$ we solve the equation
$p=\gamma_{2}/(\gamma_{1}+\gamma_{2})$ to get the pertaining $\gamma_{2}%
$-value. We vary the common size $n$ of both samples $\left(  X_{1}%
,...,X_{n}\right)  $ and $\left(  Y_{1},...,Y_{n}\right)  ,$ then for each
size, we generate $1000$ independent replicates. Our overall results are taken
as the empirical means of the results obtained through the $1000$ repetitions.
To determine the optimal number of upper order statistics (that we denote by
$k^{\ast})$ used in the computation of $\widehat{\gamma}_{1}^{\left(
H,c\right)  },$ we apply the algorithm of \cite{RT97}, page 121. The retention
level $R$ is taken as the value of the intermediate order statistic
$Z_{n-k^{\ast}:n}.$ The simulation results are summarized in Table \ref{Tab1}
for $\gamma_{1}=0.10$ and in Table \ref{Tab2} for $\gamma_{1}=0.25.$ On the
light of these results we see that, from the point of view of the rmse, the
estimation accuracy increases when the censoring percentage decreases, which
seems logical. On the other hand, we note that the sample size does not have a
significant effect on the estimation when the percentage of observed data is
high. Moreover, the estimator performs better for the smaller value of the
distortion parameter $\rho.$%

\begin{table}[h] \centering
\begin{tabular}
[c]{ccccccccc}\hline\hline
\multicolumn{9}{c}{$p=0.40$}\\\hline
$\rho$ & \multicolumn{4}{|c}{$1.00$} & \multicolumn{4}{||c}{$1.10$}\\\hline
$n$ & \multicolumn{1}{|c}{$\Pi_{\rho}(R)$} & $\widehat{\Pi}_{\rho}(R)$ &
abs.bias & rmse & \multicolumn{1}{||c}{$\Pi_{\rho}(R)$} & $\widehat{\Pi}%
_{\rho}(R)$ & abs.bias & rmse\\\hline
\multicolumn{1}{r}{$500$} & \multicolumn{1}{|c}{$0.0175$} & $0.0274$ &
$0.0099$ & $0.1372$ & \multicolumn{1}{||c}{$0.0237$} & $0.0429$ & $0.0192$ &
$0.2012$\\
\multicolumn{1}{r}{$1000$} & \multicolumn{1}{|c}{$0.0174$} & $0.0197$ &
$0.0023$ & $0.0616$ & \multicolumn{1}{||c}{$0.0236$} & $0.0339$ & $0.0102$ &
$0.0863$\\
\multicolumn{1}{r}{$1500$} & \multicolumn{1}{|c}{$0.0170$} & $0.0158$ &
$0.0012$ & $0.0146$ & \multicolumn{1}{||c}{$0.0233$} & $0.0233$ & $0.0000$ &
$0.0203$\\\hline
\multicolumn{9}{c}{$p=0.60$}\\\hline
\multicolumn{1}{r}{$500$} & \multicolumn{1}{|c}{$0.0097$} & $0.0066$ &
$0.0032$ & $0.0135$ & \multicolumn{1}{||c}{$0.0142$} & $0.0103$ & $0.0039$ &
$0.0145$\\
\multicolumn{1}{r}{$1000$} & \multicolumn{1}{|c}{$0.0095$} & $0.0037$ &
$0.0058$ & $0.0065$ & \multicolumn{1}{||c}{$0.0138$} & $0.0063$ & $0.0076$ &
$0.0091$\\
\multicolumn{1}{r}{$1500$} & \multicolumn{1}{|c}{$0.0095$} & $0.0029$ &
$0.0066$ & $0.0069$ & \multicolumn{1}{||c}{$0.0137$} & $0.0045$ & $0.0092$ &
$0.0098$\\\hline
\multicolumn{9}{c}{$p=0.80$}\\\hline
\multicolumn{1}{r}{$500$} & \multicolumn{1}{|c}{$0.0062$} & $0.0014$ &
$0.0048$ & $0.0049$ & \multicolumn{1}{||c}{$0.0093$} & $0.0026$ & $0.0067$ &
$0.0074$\\
\multicolumn{1}{r}{$1000$} & \multicolumn{1}{|c}{$0.0062$} & $0.0008$ &
$0.0054$ & $0.0055$ & \multicolumn{1}{||c}{$0.0092$} & $0.0014$ & $0.0077$ &
$0.0078$\\
\multicolumn{1}{r}{$1500$} & \multicolumn{1}{|c}{$0.0060$} & $0.0006$ &
$0.0054$ & $0.0054$ & \multicolumn{1}{||c}{$0.0090$} & $0.0010$ & $0.0079$ &
$0.0079$\\\hline\hline
\multicolumn{1}{l}{} &  &  &  &  &  &  &  &
\end{tabular}
\caption{PHP estimates based on 1000 right-censored samples of size n from Burr model with tail index $\gamma_{1}=0.10$.}\label{Tab1}%
\end{table}%
%

\begin{table}[h] \centering
\begin{tabular}
[c]{ccccccccc}\hline\hline
\multicolumn{9}{c}{$p=0.40$}\\\hline
$\rho$ & \multicolumn{4}{|c}{$1.00$} & \multicolumn{4}{||c}{$1.10$}\\\hline
$n$ & \multicolumn{1}{|c}{$\Pi_{\rho}(R)$} & $\widehat{\Pi}_{\rho}(R)$ &
abs.bias & rmse & \multicolumn{1}{||c}{$\Pi_{\rho}(R)$} & $\widehat{\Pi}%
_{\rho}(R)$ & abs.bias & rmse\\\hline
\multicolumn{1}{r}{$500$} & \multicolumn{1}{|c}{$0.0265$} & $0.0767$ &
$0.0501$ & $0.3604$ & \multicolumn{1}{||c}{$0.0410$} & $0.1148$ & $0.0738$ &
$0.8875$\\
\multicolumn{1}{r}{$1000$} & \multicolumn{1}{|c}{$0.0266$} & $0.0633$ &
$0.0368$ & $0.1602$ & \multicolumn{1}{||c}{$0.0411$} & $0.1134$ & $0.0723$ &
$0.4842$\\
\multicolumn{1}{r}{$1500$} & \multicolumn{1}{|c}{$0.0266$} & $0.0462$ &
$0.0196$ & $0.0664$ & \multicolumn{1}{||c}{$0.0409$} & $0.0632$ & $0.0223$ &
$0.0941$\\\hline
\multicolumn{9}{c}{$p=0.60$}\\\hline
\multicolumn{1}{r}{$500$} & \multicolumn{1}{|c}{$0.0196$} & $0.0229$ &
$0.0034$ & $0.0965$ & \multicolumn{1}{||c}{$0.0310$} & $0.0222$ & $0.0088$ &
$0.4596$\\
\multicolumn{1}{r}{$1000$} & \multicolumn{1}{|c}{$0.0197$} & $0.0119$ &
$0.0078$ & $0.0133$ & \multicolumn{1}{||c}{$0.0317$} & $0.0203$ & $0.0114$ &
$0.0236$\\
\multicolumn{1}{r}{$1500$} & \multicolumn{1}{|c}{$0.0199$} & $0.0093$ &
$0.0106$ & $0.0140$ & \multicolumn{1}{||c}{$0.0316$} & $0.0152$ & $0.0164$ &
$0.0196$\\\hline
\multicolumn{9}{c}{$p=0.80$}\\\hline
\multicolumn{1}{r}{$500$} & \multicolumn{1}{|c}{$0.0153$} & $0.0056$ &
$0.0097$ & $0.0118$ & \multicolumn{1}{||c}{$0.0251$} & $0.0091$ & $0.0160$ &
$0.0178$\\
\multicolumn{1}{r}{$1000$} & \multicolumn{1}{|c}{$0.0154$} & $0.0030$ &
$0.0125$ & $0.0127$ & \multicolumn{1}{||c}{$0.0254$} & $0.0054$ & $0.0200$ &
$0.0204$\\
\multicolumn{1}{r}{$1500$} & \multicolumn{1}{|c}{$0.0157$} & $0.0020$ &
$0.0136$ & $0.0137$ & \multicolumn{1}{||c}{$0.0254$} & $0.0040$ & $0.0214$ &
$0.0215$\\\hline\hline
\multicolumn{1}{l}{} &  &  &  &  &  &  &  &
\end{tabular}
\caption{PHP estimates based on 1000 right-censored samples of size n from Burr model with tail index $\gamma_{1}=0.25$.}\label{Tab2}%
\end{table}%

\section{Proof\textbf{\label{sec4}}}

\noindent Before we start the proof of the theorem, let us give a brief
introduction on some uniform empirical processes under random censoring. To
this end, we define the functions%
\[
H^{\left(  j\right)  }\left(  v\right)  :=\mathbb{P}\left(  Z\leq v,\text{
}\delta=j\right)  ,\text{ }j=0,1;\text{ }v\geq0,
\]
which have a prominent role to play in the random censorship setting. Their
empirical counterparts are defined by%
\[
H_{n}^{\left(  j\right)  }(v):=\frac{1}{n}\sum_{i=1}^{n}\mathbf{1}(Z_{i}\leq
v,\delta_{i}=j),\text{ }j=0,1;\text{ }v\geq0.
\]
In the sequel, we will use the following two empirical processes%
\[
\sqrt{n}\left(  \overline{H}_{n}^{\left(  j\right)  }(v)-\overline{H}^{\left(
j\right)  }(v)\right)  ,\text{ }j=0,1;\text{ }v\geq0,
\]
which may be represented, almost surely, by a uniform empirical process.
Indeed, let us define, for each $i=1,...,n$ with $\theta:=H^{\left(  1\right)
}\left(  \infty\right)  ,$ the following rv%
\[
U_{i}:=\delta_{i}H^{\left(  1\right)  }(Z_{i})+(1-\delta_{i})(\theta
+H^{\left(  0\right)  }(Z_{i})).
\]
From \cite{EK92}, the rv's $U_{1},...,U_{n}$ are independent and identically
distributed according to the $(0,1)$-uniform law. The empirical cdf and the
uniform empirical process based upon $U_{1},...,U_{n}$ are respectively
denoted by%
\[
\mathbb{U}_{n}(s):\mathbb{=}\frac{1}{n}\sum_{i=1}^{n}\mathbf{1}(U_{i}\leq
s)\text{ and }\alpha_{n}(s):=\sqrt{n}(\mathbb{U}_{n}(s)-s),\text{ }0\leq
s\leq1.
\]
\cite{DE96} state that almost surely%
\[
H_{n}^{\left(  0\right)  }(v)=\mathbb{U}_{n}(H^{\left(  0\right)  }%
(v)+\theta)-\mathbb{U}_{n}(\theta),\text{ for }0<H^{\left(  0\right)
}(v)<1-\theta,
\]
and%
\[
H_{n}^{\left(  1\right)  }(v)=\mathbb{U}_{n}(H^{\left(  1\right)  }(v)),\text{
for }0<H^{\left(  1\right)  }(v)<\theta.
\]
It is easy to verify that we almost surely have%
\begin{equation}
\sqrt{n}\left(  \overline{H}_{n}^{\left(  1\right)  }\left(  v\right)
-\overline{H}^{\left(  1\right)  }\left(  v\right)  \right)  =\alpha
_{n}\left(  \theta\right)  -\alpha_{n}\left(  \theta-\overline{H}^{\left(
1\right)  }\left(  v\right)  \right)  ,\text{ for }0<\overline{H}^{\left(
1\right)  }\left(  v\right)  <\theta, \label{rep-H1}%
\end{equation}
and%
\begin{equation}
\sqrt{n}\left(  \overline{H}_{n}^{\left(  0\right)  }\left(  v\right)
-\overline{H}^{\left(  0\right)  }\left(  v\right)  \right)  =-\alpha
_{n}\left(  1-\overline{H}^{\left(  0\right)  }\left(  v\right)  \right)
,\text{ for }0<\overline{H}^{\left(  0\right)  }\left(  v\right)  <1-\theta.
\label{rep-H0}%
\end{equation}
Our methodology strongly relies on the well-known Gaussian approximation given
in \ref{Prop0}. For our needs, we use the following form:%
\begin{equation}
\underset{1/n\leq s\leq1}{\sup}\frac{n^{\zeta}\left\vert \alpha_{n}%
(1-s)-B_{n}(1-s)\right\vert }{s^{1/2-\zeta}}=O_{\mathbb{P}}(1). \label{approx}%
\end{equation}
For the increments $\alpha_{n}(\theta)-\alpha_{n}(\theta-s),$ we will need an
approximation of the same type as $\left(  \ref{approx}\right)  $. Following
similar arguments, mutatis mutandis, as those used to in the proof of
assertions (2.2) of Theorem 2.1 and (2.8) of Theorem 2.2 in \cite{CHM86}, we
may show that, for every $0<\theta<1$ and $0\leq\zeta<1/4,$ we have%
\begin{equation}
\underset{1/n\leq s\leq\theta}{\sup}\frac{n^{\zeta}\left\vert \left\{
\alpha_{n}(\theta)-\alpha_{n}(\theta-s)\right\}  -\left\{  B_{n}\left(
\theta\right)  -B_{n}(\theta-s)\right\}  \right\vert }{s^{1/2-\zeta}%
}=O_{_{\mathbb{P}}}(1). \label{approx2}%
\end{equation}
The following Gaussian \ processes will be crucial to our needs:%
\begin{equation}
\mathbf{B}_{n}\left(  v\right)  :=B_{n}\left(  \theta\right)  -B_{n}\left(
\theta-\overline{H}^{\left(  1\right)  }\left(  v\right)  \right)  ,\text{ for
}0<\overline{H}^{\left(  1\right)  }\left(  v\right)  <\theta, \label{B}%
\end{equation}
and%
\begin{equation}
\mathbf{B}_{n}^{\ast}\left(  v\right)  :=\mathbf{B}_{n}\left(  v\right)
-B_{n}\left(  1-\overline{H}^{\left(  0\right)  }\left(  v\right)  \right)
,\text{ for }0<\overline{H}^{\left(  0\right)  }\left(  v\right)  <1-\theta.
\label{Bn-etoil}%
\end{equation}

\subsection{Proof of Theorem \ref{Theo}\-}

In the sequel, for two sequences of rv's, we write $V_{n}^{\left(  1\right)
}=o_{\mathbb{P}}\left(  V_{n}^{\left(  2\right)  }\right)  $ and
$V_{n}^{\left(  1\right)  }\approx V_{n}^{\left(  2\right)  },$ as
$n\rightarrow\infty,$ to say that $V_{n}^{\left(  1\right)  }/V_{n}^{\left(
2\right)  }\rightarrow0$ in probability and $V_{n}^{\left(  1\right)  }%
=V_{n}^{\left(  2\right)  }\left(  1+o_{\mathbb{P}}\left(  1\right)  \right)
$ respectively. With the premium%
\[
\Pi_{\rho}(R)=R(\overline{F}(R))^{1/\rho}%
{\displaystyle\int_{1}^{\infty}}
\left(  \frac{\overline{F}(Rx)}{\overline{F}(R)}\right)  ^{1/\rho}dx,
\]
and its estimator%
\[
\widehat{\Pi}_{\rho}(R)=\frac{\rho R}{1/\widehat{\gamma}_{1}^{(H,c)}-\rho
}\left(  \frac{R}{Z_{n-k:n}}\right)  ^{-1/\left(  \rho\widehat{\gamma}%
_{1}^{(H,c)}\right)  }\left(  \overline{F}_{n}(Z_{n-k:n})\right)  ^{1/\rho},
\]
it is easy to verify that%
\[
\sqrt{k}\frac{\widehat{\Pi}_{\rho}(R)-\Pi_{\rho}(R)}{\left(  R/h\right)
^{-1/\left(  \rho\gamma_{1}\right)  }R\left(  \overline{F}\left(  h\right)
\right)  ^{1/\rho}}=%
{\displaystyle\sum\limits_{i=1}^{5}}
S_{ni},
\]
where%
\begin{align*}
S_{n1}  &  :=\frac{\rho}{1/\widehat{\gamma}_{1}^{(H,c)}-\rho}\left(
\frac{\overline{F}(Z_{n-k:n})}{\overline{F}(h)}\right)  ^{1/\rho}\left(
\frac{\overline{F}_{n}(Z_{n-k:n})}{\overline{F}(Z_{n-k:n})}\right)  ^{1/\rho
}\\
&  \times\sqrt{k}\left\{  \left(  \frac{\left(  R/Z_{n-k:n}\right)
^{-1/\widehat{\gamma}_{1}^{(H,c)}}}{\left(  R/h\right)  ^{-1/\gamma_{1}}%
}\right)  ^{1/\rho}-1\right\}  ,
\end{align*}%
\[
S_{n2}:=\left(  \frac{\overline{F}(Z_{n-k:n})}{\overline{F}(h)}\right)
^{1/\rho}\left(  \frac{\overline{F}_{n}(Z_{n-k:n})}{\overline{F}(Z_{n-k:n}%
)}\right)  ^{1/\rho}\sqrt{k}\left\{  \frac{\rho}{1/\widehat{\gamma}%
_{1}^{(H,c)}-\rho}-\frac{\rho}{1/\gamma_{1}-\rho}\right\}  ,
\]%
\[
S_{n3}:=\frac{\rho}{1/\gamma_{1}-\rho}\left(  \frac{\overline{F}(Z_{n-k:n}%
)}{\overline{F}(h)}\right)  ^{1/\rho}\sqrt{k}\left\{  \left(  \frac
{\overline{F}_{n}(Z_{n-k:n})}{\overline{F}(Z_{n-k:n})}\right)  ^{1/\rho
}-1\right\}  ,
\]%
\[
S_{n4}:=\frac{\rho}{1/\gamma_{1}-\rho}\sqrt{k}\left\{  \left(  \frac
{\overline{F}(Z_{n-k:n})}{\overline{F}(h)}\right)  ^{1/\rho}-1\right\}  ,
\]
and%
\[
S_{n5}:=\sqrt{k}\left\{  \frac{\rho}{1/\gamma_{1}-\rho}-\frac{\left(
\overline{F}\left(  R\right)  /\overline{F}(h)\right)  ^{1/\rho}}{\left(
R/h\right)  ^{-1/\left(  \rho\gamma_{1}\right)  }}%
{\displaystyle\int_{1}^{\infty}}
\left(  \frac{\overline{F}(Rx)}{\overline{F}(R)}\right)  ^{1/\rho}dx\right\}
.
\]
We will represent the first three terms $S_{ni},$ $i=1,2,3,$ in terms of the
Gaussian processes $\mathbf{B}_{n}$ and $\mathbf{B}_{n}^{\ast}$ and we will
show that $S_{n4}\overset{\mathbb{P}}{\rightarrow}0$ while $S_{n5}$ converges
to a deterministic limit. For the first term $S_{n1},$ we have $\widehat
{\gamma}_{1}^{\left(  H,c\right)  }\overset{\mathbb{P}}{\rightarrow}\gamma
_{1}$ (see \citeauthor{BMN15}, \citeyear{BMN15}) and $Z_{n-k:n}/h\overset
{\mathbb{P}}{\rightarrow}1,$ which, in view of the regular variation of
$\overline{F},$ implies that $\overline{F}\left(  Z_{n-k:n}\right)
/\overline{F}\left(  h\right)  \overset{\mathbb{P}}{\rightarrow}1.$ Moreover,
from $\left(  \ref{p(1-p)}\right)  $ we have $\overline{F}_{n}\left(
Z_{n-k:n}\right)  /\overline{F}\left(  Z_{n-k:n}\right)  \overset{\mathbb{P}%
}{\rightarrow}1.$ It follows that $S_{n1}=S_{n1}^{(1)}+S_{n1}^{(2)},$ where%
\begin{align*}
S_{n1}^{(1)}  &  :=\left(  1+o_{\mathbb{P}}\left(  1\right)  \right)
\frac{\rho\gamma_{1}}{1-\rho\gamma_{1}}\\
&  \times\sqrt{k}\left\{  \left(  \frac{Z_{n-k:n}}{h}\right)  ^{1/\left(
\rho\widehat{\gamma}_{1}^{(H,c)}\right)  }-1\right\}  \left(  \left(  \frac
{R}{h}\right)  ^{1/\gamma_{1}-1/\widehat{\gamma}_{1}^{(H,c)}}\right)
^{1/\rho},
\end{align*}
and%
\[
S_{n1}^{(2)}:=\left(  1+o_{\mathbb{P}}\left(  1\right)  \right)  \frac
{\rho\gamma_{1}}{1-\rho\gamma_{1}}\sqrt{k}\left\{  \left(  \left(  \frac{R}%
{h}\right)  ^{1/\gamma_{1}-1/\widehat{\gamma}_{1}^{(H,c)}}\right)  ^{1/\rho
}-1\right\}  .
\]
For $S_{n1}^{(1)},$ we use the mean value theorem, the consistency of
$\widehat{\gamma}_{1}^{(H,c)}$ and the fact that $Z_{n-k:n}/h\overset
{\mathbb{P}}{\rightarrow}1,$ to have%
\[
S_{n1}^{(1)}=\left(  1+o_{\mathbb{P}}\left(  1\right)  \right)  \frac
{1}{1-\rho\gamma_{1}}\sqrt{k}\left(  \frac{Z_{n-k:n}}{h}-1\right)  .
\]
Next, we apply result (2.7) of Theorem 2.1 in \cite{BMN15} to get%
\[
S_{n1}^{(1)}=\left(  1+o_{\mathbb{P}}\left(  1\right)  \right)  \frac{\gamma
}{1-\rho\gamma_{1}}\sqrt{\frac{n}{k}}\mathbf{B}_{n}^{\ast}\left(  h\right)  .
\]
In view of the consistency and asymptotic normality of $\widehat{\gamma}%
_{1}^{(H,c)}$ and the assumption $R/h\rightarrow1,$ we show, by applying the
mean value theorem twice, that $S_{n1}^{(2)}=o_{\mathbb{P}}(1).$ Thus, we end
up with%
\begin{equation}
S_{n1}=\left(  1+o_{\mathbb{P}}\left(  1\right)  \right)  \frac{\gamma}%
{1-\rho\gamma_{1}}\sqrt{\frac{n}{k}}\mathbf{B}_{n}^{\ast}\left(  h\right)
+o_{\mathbb{P}}(1). \label{Sn1}%
\end{equation}
By similar arguments and using the mean value theorem once again, we easily
show that%
\[
S_{n2}=\left(  1+o_{\mathbb{P}}\left(  1\right)  \right)  \frac{\rho}{\left(
1-\rho\gamma_{1}\right)  ^{2}}\sqrt{k}\left(  \widehat{\gamma}_{1}%
^{(H,c)}-\gamma_{1}\right)  ,
\]%
\[
S_{n3}=\left(  1+o_{\mathbb{P}}\left(  1\right)  \right)  \frac{\gamma_{1}%
}{1-\rho\gamma_{1}}\sqrt{k}\left(  \frac{\overline{F}_{n}(Z_{n-k:n}%
)}{\overline{F}(Z_{n-k:n})}-1\right)  ,
\]
and%
\[
S_{n4}=\left(  1+o_{\mathbb{P}}\left(  1\right)  \right)  \frac{\gamma_{1}%
}{1-\rho\gamma_{1}}\sqrt{k}\left\{  \frac{\overline{F}(Z_{n-k:n})}%
{\overline{F}(h)}-1\right\}  .
\]
By applying result (2.9) of Theorem 2.1 in \cite{BMN15} we get, after a change
of variables, that%
\begin{align}
S_{n2}  &  =\left(  1+o_{\mathbb{P}}\left(  1\right)  \right)  \frac{\rho
}{\left(  1-\rho\gamma_{1}\right)  ^{2}}\left\{  \frac{1}{p}\sqrt{\frac{n}{k}%
}\int_{1}^{\infty}v^{-1}\mathbf{B}_{n}^{\ast}\left(  hv\right)  dv-\frac
{\gamma_{1}}{p}\sqrt{\frac{n}{k}}\mathbf{B}_{n}\left(  h\right)  \right\}
\nonumber\\
&  +\left(  1+o_{\mathbb{P}}\left(  1\right)  \right)  \frac{\rho\sqrt{k}%
A_{1}\left(  h\right)  }{\left(  1-p\tau_{1}\right)  \left(  1-\rho\gamma
_{1}\right)  ^{2}}. \label{Sn2}%
\end{align}
From Proposition \ref{Prop1}, we infer that%
\begin{equation}
S_{n3}=\left(  1+o_{\mathbb{P}}\left(  1\right)  \right)  \frac{\gamma_{1}%
}{1-\rho\gamma_{1}}\left(  \sqrt{\dfrac{n}{k}}\mathbf{B}_{n}\left(  h\right)
+\sqrt{\dfrac{k}{n}}\Delta_{n}\right)  +o_{\mathbb{P}}\left(  1\right)  .
\label{Sn3}%
\end{equation}
Now, we decompose $S_{n4}$ into the sum of two terms%
\[
S_{n4}^{\left(  1\right)  }:=\left(  1+o_{\mathbb{P}}\left(  1\right)
\right)  \frac{\gamma_{1}}{1-\rho\gamma_{1}}\sqrt{k}\left\{  \frac
{\overline{F}(Z_{n-k:n})}{\overline{F}(h)}-\left(  \frac{Z_{n-k:n}}{h}\right)
^{-1/\gamma_{1}}\right\}  ,
\]
and%
\[
S_{n4}^{\left(  2\right)  }:=\left(  1+o_{\mathbb{P}}\left(  1\right)
\right)  \frac{\gamma_{1}}{1-\rho\gamma_{1}}\sqrt{k}\left\{  \left(
\frac{Z_{n-k:n}}{h}\right)  ^{-1/\gamma_{1}}-1\right\}  .
\]
The second-order condition $\left(  \ref{second-order}\right)  $ of
$\overline{F}$ and the fact that $Z_{n-k:n}/h\overset{\mathbb{P}}{\rightarrow
}1$ yield that%
\[
S_{n4}^{\left(  1\right)  }=o_{\mathbb{P}}\left(  \sqrt{k}A_{1}\left(
h\right)  \right)  =o_{\mathbb{P}}\left(  1\right)  .
\]
For $S_{n4}^{\left(  2\right)  },$ we, once again, apply the mean value
theorem (with $Z_{n-k:n}/h\overset{\mathbb{P}}{\rightarrow}1)$ then we use
result (2.7) of Theorem 2.1 in \cite{BMN15} to get%
\[
S_{n4}^{\left(  2\right)  }=-\left(  1+o_{\mathbb{P}}\left(  1\right)
\right)  \frac{\gamma}{1-\rho\gamma_{1}}\sqrt{\frac{n}{k}}\mathbf{B}_{n}%
^{\ast}\left(  h\right)  .
\]
Consequently, we have%
\begin{equation}
S_{n4}=-\left(  1+o_{\mathbb{P}}\left(  1\right)  \right)  \frac{\gamma
}{1-\rho\gamma_{1}}\sqrt{\frac{n}{k}}\mathbf{B}_{n}^{\ast}\left(  h\right)
+o_{\mathbb{P}}\left(  1\right)  . \label{Sn4}%
\end{equation}
For the last term $S_{n5},$ we start by decomposing it into the sum of%
\[
S_{n5}^{(1)}:=-\frac{\rho\gamma_{1}}{1-\rho\gamma_{1}}\frac{1}{\left(
R/h\right)  ^{-1/\left(  \rho\gamma_{1}\right)  }}\sqrt{k}\left\{  \left(
\frac{\overline{F}(R)}{\overline{F}(h)}\right)  ^{1/\rho}-\left(  \left(
\frac{R}{h}\right)  ^{-1/\gamma_{1}}\right)  ^{1/\rho}\right\}  ,
\]
and%
\[
S_{n5}^{(2)}:=-\left(  \frac{\left(  \overline{F}\left(  R\right)
/\overline{F}(h)\right)  }{\left(  R/h\right)  ^{-1/\gamma_{1}}}\right)
^{1/\rho}\sqrt{k}%
{\displaystyle\int_{1}^{\infty}}
\left(  \left(  \frac{\overline{F}(Rx)}{\overline{F}(R)}\right)  ^{1/\rho
}-\left(  x^{-1/\gamma_{1}}\right)  ^{1/\rho}\right)  dx.
\]
By similar arguments as those used for $S_{n4}^{\left(  1\right)  },$ we show
that (here we use the assumption that $R/h\rightarrow1)$%
\[
S_{n5}^{(1)}=o_{\mathbb{P}}\left(  \sqrt{k}A_{1}\left(  h\right)  \right)
=o_{\mathbb{P}}\left(  1\right)  .
\]
For $S_{n5}^{(2)},$ we first apply the mean value theorem to have%
\[
S_{n5}^{(2)}=-\frac{1}{\rho}\sqrt{k}%
{\displaystyle\int_{1}^{\infty}}
\left(  \frac{\overline{F}(Rx)}{\overline{F}(R)}-x^{-1/\gamma_{1}}\right)
\zeta^{1/\rho-1}(x)dx,
\]
where $\zeta$ lies between $\overline{F}(Rx)/\overline{F}(R)$ and
$x^{-1/\gamma_{1}}.$ Then we use Potter's inequalities, given in assertion 5
of Proposition B.1.9 in \cite{deHF06}, to get%
\[
S_{n5}^{(2)}=\left(  1+o\left(  1\right)  \right)  \frac{\sqrt{k}A_{1}\left(
h\right)  }{\rho\left(  \gamma_{1}+\tau_{1}+\rho-2\right)  \left(
2-\rho-\gamma_{1}\right)  }.
\]
Therefore%
\begin{equation}
S_{n5}=\left(  1+o\left(  1\right)  \right)  \frac{\sqrt{k}A_{1}\left(
h\right)  }{\rho\left(  \gamma_{1}+\tau_{1}+\rho-2\right)  \left(
2-\rho-\gamma_{1}\right)  }+o_{\mathbb{P}}\left(  1\right)  . \label{Sn5}%
\end{equation}
Finally, by gathering results $\left(  \ref{Sn1}\right)  ,$ $\left(
\ref{Sn2}\right)  ,$ $\left(  \ref{Sn3}\right)  ,$ $\left(  \ref{Sn4}\right)
$ and $\left(  \ref{Sn5}\right)  ,$ we obtain the following asymptotic
representation to the premium estimator:%
\begin{align}
\sqrt{k}\frac{\widehat{\Pi}_{\rho}(R)-\Pi_{\rho}(R)}{\left(  R/h\right)
^{-1/\rho\gamma_{1}}R\left(  \overline{F}\left(  h\right)  \right)  ^{1/\rho
}}  &  =o_{\mathbb{P}}\left(  1\right)  +\frac{\gamma_{1}}{1-\rho\gamma_{1}%
}\sqrt{\dfrac{k}{n}}\Delta_{n}+\frac{1}{1-\rho\gamma_{1}}\sqrt{\frac{n}{k}%
}\Gamma_{n}\label{prim-approx}\\
&  +\left\{  \dfrac{\rho\sqrt{k}A_{1}\left(  h\right)  }{\left(  1-p\tau
_{1}\right)  \left(  1-\rho\gamma_{1}\right)  ^{2}}+\dfrac{\sqrt{k}%
A_{1}\left(  h\right)  }{\rho\left(  \gamma_{1}+\tau_{1}+\rho-2\right)
\left(  2-\rho-\gamma_{1}\right)  }\right\}  ,\nonumber
\end{align}
where $\Delta_{n}$ is as defined in \ref{delta} and%
\[
\Gamma_{n}:=\gamma_{1}\left(  1-\frac{\rho}{p\left(  1-\rho\gamma_{1}\right)
}\right)  \mathbf{B}_{n}\left(  h\right)  +\frac{\rho}{p\left(  1-\rho
\gamma_{1}\right)  }\int_{1}^{\infty}v^{-1}\mathbf{B}_{n}^{\ast}\left(
hv\right)  dv.
\]
From $\left(  \ref{prim-approx}\right)  ,$ we deduce that $\sqrt{k}\left(
\widehat{\Pi}_{\rho}(R)-\Pi_{\rho}(R)\right)  /\left(  \left(  R/h\right)
^{-1/\rho\gamma_{1}}R\left(  \overline{F}\left(  h\right)  \right)  ^{1/\rho
}\right)  $ is asymptotically Gaussian with mean%
\[
\left\{  \dfrac{\rho}{\left(  1-p\tau_{1}\right)  \left(  1-\rho\gamma
_{1}\right)  ^{2}}+\dfrac{1}{\rho\left(  \gamma_{1}+\tau_{1}+\rho-2\right)
\left(  2-\rho-\gamma_{1}\right)  }\right\}  \lim_{n\rightarrow\infty}\sqrt
{k}A_{1}\left(  h\right)  =\mu,
\]
and variance%
\[
\lim_{n\rightarrow\infty}\mathbf{E}\left[  \frac{\gamma_{1}}{1-\rho\gamma_{1}%
}\sqrt{\dfrac{k}{n}}\Delta_{n}+\frac{1}{1-\rho\gamma_{1}}\sqrt{\frac{n}{k}%
}\Gamma_{n}\right]  ^{2}.
\]
Note that from the covariance structure in \cite{C96}, page 2768, we have the
following useful formulas:%
\begin{equation}
\left\{
\begin{tabular}
[c]{l}%
$\mathbf{E}\left[  \mathbf{B}_{n}\left(  u\right)  \mathbf{B}_{n}\left(
v\right)  \right]  =\min\left(  \overline{H}^{\left(  1\right)  }\left(
u\right)  ,\overline{H}^{\left(  1\right)  }\left(  v\right)  \right)
-\overline{H}^{\left(  1\right)  }\left(  u\right)  \overline{H}^{\left(
1\right)  }\left(  v\right)  ,\smallskip$\\
$\mathbf{E}\left[  \mathbf{B}_{n}^{\ast}\left(  u\right)  \mathbf{B}_{n}%
^{\ast}\left(  v\right)  \right]  =\min\left(  \overline{H}\left(  u\right)
,\overline{H}\left(  v\right)  \right)  -\overline{H}\left(  u\right)
\overline{H}\left(  v\right)  ,\smallskip$\\
$\mathbf{E}\left[  \mathbf{B}_{n}\left(  u\right)  \mathbf{B}_{n}^{\ast
}\left(  v\right)  \right]  =\min\left(  \overline{H}^{\left(  1\right)
}\left(  u\right)  ,\overline{H}^{\left(  1\right)  }\left(  v\right)
\right)  -\overline{H}^{\left(  1\right)  }\left(  u\right)  \overline
{H}\left(  v\right)  .$%
\end{tabular}
\ \right.  \label{covariances}%
\end{equation}
After elementary but very tedious computations, using these formulas with
l'H\^{o}pital's rule, we get as $n\rightarrow\infty,$%
\[%
{\displaystyle\int_{0}^{h}}
\frac{\mathbf{E}\left[  \mathbf{B}_{n}\left(  u\right)  \mathbf{B}_{n}\left(
h\right)  \right]  }{\overline{H}^{2}\left(  u\right)  }d\overline{H}\left(
u\right)  \rightarrow-p,\text{ }%
{\displaystyle\int_{0}^{h}}
\frac{\mathbf{E}\left[  \mathbf{B}_{n}\left(  h\right)  \mathbf{B}_{n}^{\ast
}\left(  u\right)  \right]  }{\overline{H}^{2}\left(  u\right)  }d\overline
{H}^{\left(  1\right)  }\left(  u\right)  \rightarrow-p^{2},
\]%
\[%
{\displaystyle\int_{0}^{h}}
{\displaystyle\int_{1}^{\infty}}
\dfrac{\mathbf{E}\left[  \mathbf{B}_{n}\left(  v\right)  \mathbf{B}_{n}^{\ast
}\left(  hu\right)  \right]  }{u\overline{H}^{2}\left(  v\right)
}dud\overline{H}\left(  v\right)  \rightarrow-p\gamma,
\]%
\[%
{\displaystyle\int_{0}^{h}}
{\displaystyle\int_{1}^{\infty}}
\dfrac{\mathbf{E}\left[  \mathbf{B}_{n}^{\ast}\left(  v\right)  \mathbf{B}%
_{n}^{\ast}\left(  hu\right)  \right]  }{u\overline{H}^{2}\left(  v\right)
}dud\overline{H}^{\left(  1\right)  }\left(  v\right)  \rightarrow-p\gamma,
\]%
\[
\dfrac{k}{n}%
{\displaystyle\int_{0}^{h}}
{\displaystyle\int_{0}^{h}}
\dfrac{\mathbf{E}\left[  \mathbf{B}_{n}\left(  u\right)  \mathbf{B}_{n}\left(
v\right)  \right]  }{\overline{H}^{2}(u)\overline{H}^{2}\left(  v\right)
}d\overline{H}(u)d\overline{H}\left(  v\right)  \rightarrow2p,
\]%
\[
\dfrac{k}{n}%
{\displaystyle\int_{0}^{h}}
{\displaystyle\int_{0}^{h}}
\dfrac{\mathbf{E}\left[  \mathbf{B}_{n}^{\ast}\left(  u\right)  \mathbf{B}%
_{n}^{\ast}\left(  v\right)  \right]  }{\overline{H}^{2}(u)\overline{H}%
^{2}\left(  v\right)  }d\overline{H}^{\left(  1\right)  }(u)d\overline
{H}^{\left(  1\right)  }\left(  v\right)  \rightarrow2p^{2},
\]
and%
\[
\dfrac{k}{n}%
{\displaystyle\int_{0}^{h}}
{\displaystyle\int_{0}^{h}}
\dfrac{\mathbf{E}\left[  \mathbf{B}_{n}\left(  u\right)  \mathbf{B}_{n}^{\ast
}\left(  v\right)  \right]  }{\overline{H}^{2}(u)\overline{H}^{2}\left(
v\right)  }d\overline{H}(u)d\overline{H}^{\left(  1\right)  }\left(  v\right)
\rightarrow2p^{2},
\]
Using the results above with some further calculations leads to $\sigma^{2}%
.$\textbf{\hfill}$\mathbf{\Box}$

\section{\textbf{Appendix\label{sec5}}}

\noindent The following proposition consists in Corollary 2.1 of \cite{CHM86}.

\begin{proposition}
\label{Prop0}There exists a probability space $(\Omega,\mathcal{A}%
,\mathbb{P)}$ with independent $\left(  0,1\right)  $-uniform rv's $U_{1},$
$U_{2},...$ and a sequence of Brownian bridges $\left\{  B_{i}(s);\text{
}0\leq s\leq1\right\}  $ $\left(  i=1,2,...\right)  $ such that, for every
$0<\lambda<\infty,$ we have as $n\rightarrow\infty$%
\[
\underset{\lambda/n\leq s\leq1}{\sup}\frac{n^{\zeta}\left\vert \alpha
_{n}(s)-B_{n}(s)\right\vert }{s^{1/2-\zeta}}=\left\{
\begin{array}
[c]{ll}%
O_{\mathbb{P}}(\log n) & \text{when }\zeta=\dfrac{1}{4},\\
O_{\mathbb{P}}(1) & \text{when }0\leq\zeta<\dfrac{1}{4},
\end{array}
\right.
\]%
\[
\underset{0\leq s\leq1-\lambda/n}{\sup}\frac{n^{\zeta}\left\vert \alpha
_{n}(s)-B_{n}(s)\right\vert }{(1-s)^{1/2-\zeta}}=\left\{
\begin{array}
[c]{ll}%
O_{\mathbb{P}}(\log n) & \text{when }\zeta=\dfrac{1}{4},\\
O_{\mathbb{P}}(1) & \text{when }0\leq\zeta<\dfrac{1}{4},
\end{array}
\right.
\]
and%
\[
\underset{\lambda/n\leq s\leq1-\lambda/n}{\sup}\frac{n^{\zeta}\left\vert
\alpha_{n}(s)-B_{n}(s)\right\vert }{\left(  s(1-s)\right)  ^{1/2-\zeta}%
}=\left\{
\begin{array}
[c]{ll}%
O_{\mathbb{P}}(\log n) & \text{when }\zeta=\dfrac{1}{4},\\
O_{\mathbb{P}}(1) & \text{when }0\leq\zeta<\dfrac{1}{4}.
\end{array}
\right.
\]

\end{proposition}

\begin{proof}
See \cite{CHM86}, page 48.
\end{proof}

\noindent In the next basic proposition, we provide an asymptotic
representation to the Kaplan-Meier product limit estimator in $Z_{n-k:n}.$
This result will be of prime importance in the study of the limiting behaviors
of many statistics based on censored data exhibiting extreme values.

\begin{proposition}
\label{Prop1}Assume that all second-order conditions $\left(
\ref{second-order}\right)  $ hold. Let $k=k_{n}$ be an integer sequence
satisfying, in addition to $\left(  \ref{k}\right)  $ $\sqrt{k}A_{j}\left(
h\right)  =O(1),$ for $j=1,2,$ as $n\rightarrow\infty$. Then there exists a
sequence of Brownian bridges $\left\{  B_{n}(s);\text{ }0\leq s\leq1\right\}
$ such that%
\[
\sqrt{k}\left\{  \frac{\overline{F}_{n}\left(  Z_{n-k:n}\right)  }%
{\overline{F}\left(  Z_{n-k:n}\right)  }-1\right\}  =\sqrt{\frac{n}{k}%
}\mathbf{B}_{n}\left(  h\right)  +\sqrt{\frac{k}{n}}\Delta_{n}+o_{\mathbb{P}%
}\left(  1\right)  ,
\]
where%
\begin{equation}
\Delta_{n}:=\int_{0}^{h}\frac{\mathbf{B}_{n}\left(  v\right)  }{\overline
{H}^{2}\left(  v\right)  }d\overline{H}\left(  v\right)  -\int_{0}^{h}%
\frac{\mathbf{B}_{n}^{\ast}\left(  v\right)  }{\overline{H}^{2}\left(
v\right)  }d\overline{H}^{\left(  1\right)  }\left(  v\right)  , \label{delta}%
\end{equation}
with $\mathbf{B}_{n}\left(  v\right)  $ and $\mathbf{B}_{n}^{\ast}\left(
v\right)  $ respectively defined in $\left(  \ref{B}\right)  $ and $\left(
\ref{Bn-etoil}\right)  .$ Consequently,%
\begin{equation}
\sqrt{k}\left\{  \frac{\overline{F}_{n}\left(  Z_{n-k:n}\right)  }%
{\overline{F}\left(  Z_{n-k:n}\right)  }-1\right\}  \overset{d}{\rightarrow
}\mathcal{N}\left(  0,p(1-p)\right)  ,\text{ as }n\rightarrow\infty,
\label{p(1-p)}%
\end{equation}

\end{proposition}

\begin{proof}
In view of Proposition 5 of \cite{C96}, combined with equation $\left(
4.9\right)  $ in the same reference, we have for any $x\leq Z_{n-k:n},$%
\begin{align*}
&  \frac{\overline{F}_{n}\left(  x\right)  -\overline{F}\left(  x\right)
}{\overline{F}\left(  x\right)  }\\
&  =\int_{0}^{x}\frac{d\left(  \overline{H}_{n}^{\left(  1\right)  }\left(
v\right)  -\overline{H}^{\left(  1\right)  }\left(  v\right)  \right)
}{\overline{H}\left(  v\right)  }-\int_{0}^{x}\frac{\overline{H}_{n}\left(
v\right)  -\overline{H}\left(  v\right)  }{\overline{H}^{2}\left(  v\right)
}d\overline{H}^{\left(  1\right)  }\left(  v\right)  +O_{\mathbb{P}}\left(
\frac{1}{k}\right)  .
\end{align*}
Upon integrating the first integral by parts, we get%
\begin{align}
&  \frac{\overline{F}_{n}\left(  x\right)  -\overline{F}\left(  x\right)
}{\overline{F}\left(  x\right)  }\label{ratio}\\
&  =-\left(  \overline{H}_{n}^{\left(  1\right)  }\left(  0\right)
-\overline{H}^{\left(  1\right)  }\left(  0\right)  \right)  +\frac
{\overline{H}_{n}^{\left(  1\right)  }\left(  x\right)  -\overline{H}^{\left(
1\right)  }\left(  x\right)  }{\overline{H}\left(  x\right)  }\nonumber\\
&  +\int_{0}^{x}\frac{\overline{H}_{n}^{\left(  1\right)  }\left(  v\right)
-\overline{H}^{\left(  1\right)  }\left(  v\right)  }{\overline{H}^{2}\left(
v\right)  }d\overline{H}\left(  v\right)  -\int_{0}^{x}\frac{\overline{H}%
_{n}\left(  v\right)  -\overline{H}\left(  v\right)  }{\overline{H}^{2}\left(
v\right)  }d\overline{H}^{\left(  1\right)  }\left(  v\right)  +O_{\mathbb{P}%
}\left(  \frac{1}{k}\right)  .\nonumber
\end{align}
Recall that%
\[
\sqrt{n}\left(  \overline{H}_{n}\left(  v\right)  -\overline{H}\left(
v\right)  \right)  =\sqrt{n}\left(  \overline{H}_{n}^{1}\left(  v\right)
-\overline{H}^{1}\left(  v\right)  \right)  +\sqrt{n}\left(  \overline{H}%
_{n}^{0}\left(  v\right)  -\overline{H}^{0}\left(  v\right)  \right)  ,
\]
which by representations $\left(  \ref{rep-H1}\right)  $ and $\left(
\ref{rep-H0}\right)  $ becomes%
\[
\sqrt{n}\left(  \overline{H}_{n}\left(  v\right)  -\overline{H}\left(
v\right)  \right)  =\left(  \alpha_{n}\left(  \theta\right)  -\alpha
_{n}\left(  \theta-\overline{H}^{\left(  1\right)  }\left(  v\right)  \right)
\right)  -\alpha_{n}\left(  1-\overline{H}^{\left(  0\right)  }\left(
v\right)  \right)  .
\]
On the other hand, by the classical central limit theorem, we have
$\overline{H}_{n}^{\left(  1\right)  }\left(  0\right)  -\overline{H}^{\left(
1\right)  }\left(  0\right)  =O_{\mathbb{P}}\left(  n^{-1/2}\right)  .$ Using
these results in $\left(  \ref{ratio}\right)  $ and then multiplying by
$\sqrt{k},$ we get%
\begin{align*}
&  \sqrt{k}\frac{\overline{F}_{n}\left(  x\right)  -\overline{F}\left(
x\right)  }{\overline{F}\left(  x\right)  }\\
&  =O_{\mathbb{P}}\left(  \sqrt{\frac{k}{n}}\right)  +O_{\mathbb{P}}\left(
\frac{1}{\sqrt{k}}\right)  +\sqrt{\frac{k}{n}}\frac{\alpha_{n}\left(
\theta\right)  -\alpha_{n}\left(  \theta-\overline{H}^{\left(  1\right)
}\left(  x\right)  \right)  }{\overline{H}\left(  x\right)  }\\
&  +\sqrt{\frac{k}{n}}\int_{0}^{x}\frac{\alpha_{n}\left(  \theta\right)
-\alpha_{n}\left(  \theta-\overline{H}^{\left(  1\right)  }\left(  v\right)
\right)  }{\overline{H}^{2}\left(  v\right)  }d\overline{H}\left(  v\right) \\
&  -\sqrt{\frac{k}{n}}\int_{0}^{x}\frac{\alpha_{n}\left(  \theta\right)
-\alpha_{n}\left(  \theta-\overline{H}^{\left(  1\right)  }\left(  v\right)
\right)  -\alpha_{n}\left(  1-\overline{H}^{\left(  0\right)  }\left(
v\right)  \right)  }{\overline{H}^{2}\left(  v\right)  }d\overline{H}^{\left(
1\right)  }\left(  v\right)  .
\end{align*}
The Gaussian approximations $\left(  \ref{approx}\right)  $ and $\left(
\ref{approx2}\right)  ,$ in $x=Z_{n-k:n},$ and the facts that $\sqrt{k/n}$ and
$1/\sqrt{k}$ tend to zero as $n\rightarrow\infty,$ lead to%
\begin{align*}
&  \sqrt{k}\frac{\overline{F}_{n}\left(  Z_{n-k:n}\right)  -\overline
{F}\left(  Z_{n-k:n}\right)  }{\overline{F}\left(  Z_{n-k:n}\right)  }\\
&  =\sqrt{\frac{n}{k}}\mathbf{B}_{n}\left(  Z_{n-k:n}\right)  +\sqrt{\frac
{k}{n}}\int_{0}^{Z_{n-k:n}}\frac{\mathbf{B}_{n}\left(  v\right)  }%
{\overline{H}^{2}\left(  v\right)  }d\overline{H}\left(  v\right)
-\sqrt{\frac{k}{n}}\int_{0}^{Z_{n-k:n}}\frac{\mathbf{B}_{n}^{\ast}\left(
v\right)  }{\overline{H}^{2}\left(  v\right)  }d\overline{H}^{\left(
1\right)  }\left(  v\right)  +o_{\mathbb{P}}\left(  1\right)  .
\end{align*}
Applying Lemma \ref{Lem1} completes the proof. The asymptotic normality
property is straightforward. For the variance computation, we use the
covariance formulas $\left(  \ref{covariances}\right)  $ and the results at
the end of Section \ref{sec4}.
\end{proof}

\begin{lemma}
\label{Lem1}Assume that the second-order conditions of regular variation
$\left(  \ref{second-order}\right)  $ and let $k:=k_{n}$ be an integer
sequence satisfying $\left(  \ref{k}\right)  $. Then%
\[%
\begin{tabular}
[c]{l}%
$\left(  i\right)  \text{ }\sqrt{\dfrac{k}{n}}%
{\displaystyle\int_{h}^{Z_{n-k:n}}}
\dfrac{\mathbf{B}_{n}\left(  v\right)  }{\overline{H}^{2}\left(  v\right)
}d\overline{H}\left(  v\right)  =o_{\mathbb{P}}\left(  1\right)  .\medskip$\\
$\left(  ii\right)  \text{ }\sqrt{\dfrac{k}{n}}%
{\displaystyle\int_{h}^{Z_{n-k:n}}}
\dfrac{\mathbf{B}_{n}^{\ast}\left(  v\right)  }{\overline{H}^{2}\left(
v\right)  }d\overline{H}^{\left(  1\right)  }\left(  v\right)  =o_{\mathbb{P}%
}\left(  1\right)  .\medskip$\\
$(iii)$ $\sqrt{\dfrac{n}{k}}\left\{  \mathbf{B}_{n}\left(  Z_{n-k:n}\right)
-\mathbf{B}_{n}\left(  h\right)  \right\}  =o_{\mathbb{P}}\left(  1\right)
\medskip$\\
$(iv)$ $\sqrt{\dfrac{n}{k}}\left\{  \mathbf{B}_{n}^{\ast}\left(
Z_{n-k:n}\right)  -\mathbf{B}_{n}^{\ast}\left(  h\right)  \right\}
=o_{\mathbb{P}}\left(  1\right)  .$%
\end{tabular}
\ \ \ \ \ \ \ \ \ \ \
\]

\end{lemma}

\begin{proof}
We begin by proving the first assertion. For fixed $0<\eta,$ $\varepsilon<1,$
we have%
\begin{align*}
&  \mathbb{P}\left(  \left\vert \sqrt{\frac{k}{n}}\int_{h}^{Z_{n-k:n}%
}\mathbf{B}_{n}\left(  v\right)  \frac{d\overline{H}\left(  v\right)
}{\overline{H}^{2}\left(  v\right)  }\right\vert >\eta\right) \\
&  \leq\mathbb{P}\left(  \left\vert \frac{Z_{n-k:n}}{h}-1\right\vert
>\varepsilon\right)  +\mathbb{P}\left(  \left\vert \sqrt{\frac{k}{n}}\int
_{h}^{\left(  1+\varepsilon\right)  h}\mathbf{B}_{n}\left(  v\right)
\frac{d\overline{H}\left(  v\right)  }{\overline{H}^{2}\left(  v\right)
}\right\vert >\eta\right)  .
\end{align*}
It is clear that the first term the right-hand side tends to zero as
$n\rightarrow\infty.$ Then, it remains to show that the second one goes to
zero as well.\ Indeed, observe that%
\[
\mathbf{E}\left\vert \sqrt{\frac{k}{n}}\int_{h}^{\left(  1+\varepsilon\right)
h}\mathbf{B}_{n}\left(  v\right)  \frac{d\overline{H}\left(  v\right)
}{\overline{H}^{2}\left(  v\right)  }\right\vert \leq-\sqrt{\frac{k}{n}}%
\int_{h}^{\left(  1+\varepsilon\right)  h}\mathbf{E}\left\vert \mathbf{B}%
_{n}\left(  v\right)  \right\vert \frac{d\overline{H}\left(  v\right)
}{\overline{H}^{2}\left(  v\right)  }.
\]
From the first result of $\left(  \ref{covariances}\right)  ,$ we have
$\mathbf{E}\left\vert \mathbf{B}_{n}\left(  v\right)  \right\vert \leq
\sqrt{\overline{H}^{1}\left(  v\right)  }.$ Then%
\[
\mathbf{E}\left\vert \sqrt{\frac{k}{n}}\int_{h}^{\left(  1+\varepsilon\right)
h}\mathbf{B}_{n}\left(  v\right)  \frac{d\overline{H}\left(  v\right)
}{\overline{H}^{2}\left(  v\right)  }\right\vert \leq-\sqrt{\frac{k}{n}}%
\int_{h}^{\left(  1+\varepsilon\right)  h}\sqrt{\overline{H}^{1}\left(
v\right)  }\frac{d\overline{H}\left(  v\right)  }{\overline{H}^{2}\left(
v\right)  },
\]
which, in turn, is less than or equal to%
\[
\sqrt{\frac{k}{n}}\sqrt{\overline{H}^{\left(  1\right)  }\left(  h\right)
}\left(  \frac{1}{\overline{H}\left(  \left(  1+\varepsilon\right)  h\right)
}-\frac{1}{\overline{H}\left(  h\right)  }\right)  .
\]
Since $\overline{H}\left(  h\right)  =k/n,$ then this may be rewritten into%
\[
\sqrt{\frac{\overline{H}^{\left(  1\right)  }\left(  h\right)  }{\overline
{H}\left(  h\right)  }}\left(  \frac{\overline{H}\left(  h\right)  }%
{\overline{H}\left(  \left(  1+\varepsilon\right)  h\right)  }-1\right)  .
\]
Since $\overline{H}^{\left(  1\right)  }\left(  h\right)  \sim p\overline
{H}\left(  h\right)  $ and $\overline{H}\in\mathcal{RV}_{\left(
-1/\gamma\right)  },$ then the previous quantity tends to $p^{1/2}\left(
\left(  1+\varepsilon\right)  ^{1/\gamma}-1\right)  $ as $n\rightarrow\infty.$
Being arbitrary, $\varepsilon$ may be chosen small enough so that this limit
be zero.$\ $By similar arguments, we also show assertion $\left(  ii\right)
,$ therefore we omit the details. The last two assertions are shown following
the same technique, that we use to prove $(iv).$ Notice that, from the
definition of $\mathbf{B}_{n}^{\ast}\left(  v\right)  $ and the second
covariance formula in $\left(  \ref{covariances}\right)  ,$
\[
\left\{  \mathbf{B}_{n}^{\ast}\left(  v\right)  ;\text{ }v\geq0\right\}
\overset{d}{=}\left\{  \mathcal{B}_{n}\left(  \overline{H}\left(  v\right)
\right)  ;\text{ }v\geq0\right\}  ,
\]
where $\left\{  \mathcal{B}_{n}\left(  s\right)  ;\text{ }0\leq s\leq
1\right\}  $ is a sequence of standard Brownian bridges. Hence%
\[
\sqrt{\dfrac{n}{k}}\left\{  \mathbf{B}_{n}^{\ast}\left(  Z_{n-k:n}\right)
-\mathbf{B}_{n}^{\ast}\left(  h\right)  \right\}  \overset{d}{=}\sqrt
{\dfrac{n}{k}}\left\{  \mathcal{B}_{n}\left(  \overline{H}\left(
Z_{n-k:n}\right)  \right)  -\mathcal{B}_{n}\left(  \overline{H}\left(
h\right)  \right)  \right\}  .
\]
Let $\left\{  \mathcal{W}_{n}\left(  t\right)  ;\text{ }0\leq s\leq1\right\}
$ be a sequence of standard Wiener processes such that $\mathcal{B}_{n}\left(
t\right)  =\mathcal{W}_{n}\left(  t\right)  -t\mathcal{W}_{n}\left(  1\right)
.$ Then $\sqrt{n/k}\left\{  \mathbf{B}_{n}^{\ast}\left(  Z_{n-k:n}\right)
-\mathbf{B}_{n}^{\ast}\left(  h\right)  \right\}  $ equals in distribution to%
\[
\sqrt{\dfrac{n}{k}}\left(  \left\{  \mathcal{W}_{n}\left(  \overline{H}\left(
Z_{n-k:n}\right)  \right)  -\mathcal{W}_{n}\left(  \overline{H}\left(
h\right)  \right)  \right\}  -\left\{  \overline{H}\left(  Z_{n-k:n}\right)
-\overline{H}\left(  h\right)  \right\}  \mathcal{W}_{n}\left(  1\right)
\right)  .
\]
By using the facts that $\overline{H}\left(  h\right)  =k/n$ and $\overline
{H}\left(  Z_{n-k:n}\right)  /\overline{H}\left(  h\right)  \approx1,$ we get%
\[
\sqrt{\dfrac{n}{k}}\left(  \overline{H}\left(  Z_{n-k:n}\right)  -\overline
{H}\left(  h\right)  \right)  =\sqrt{\dfrac{k}{n}}\left(  \frac{\overline
{H}\left(  Z_{n-k:n}\right)  }{\overline{H}\left(  h\right)  }-1\right)
=o_{\mathbb{P}}\left(  1\right)  .
\]
Next we show that
\[
\vartheta_{n}:=\sqrt{\dfrac{n}{k}}\left\{  \mathcal{W}_{n}\left(  \overline
{H}\left(  Z_{n-k:n}\right)  \right)  -\mathcal{W}_{n}\left(  \overline
{H}\left(  h\right)  \right)  \right\}  =o_{\mathbb{P}}\left(  1\right)  .
\]
Let $\eta>0$ be a fixed real number and show that $\mathbb{P}\left(
\left\vert \vartheta_{n}\right\vert >\eta\right)  \rightarrow0,$ as
$n\rightarrow\infty.$ Since $Z_{n-k:n}/h\overset{\mathbb{P}}{\rightarrow}1,$
then for an arbitrary $\epsilon>0$ and sufficiently large $n,$ the probability
of $A_{n}\left(  \epsilon\right)  :=\left\{  \left\vert Z_{n-k:n}%
/h-1\right\vert \leq\epsilon\right\}  $ is close to $1.$ Next, we will use the
following useful inequality: $\mathbb{P}\left(  \left\vert \vartheta
_{n}\right\vert >\eta\right)  \leq\mathbb{P}\left\{  \left\vert \vartheta
_{n}\right\vert >\eta,\text{ }A_{n}\left(  \epsilon\right)  \right\}
+\mathbb{P}\left\{  A_{n}^{c}\left(  \epsilon\right)  \right\}  ,$ where
$A_{n}^{c}\left(  \epsilon\right)  $ denotes the complement set of
$A_{n}\left(  \epsilon\right)  .$ It is easy to verify that $\vartheta_{n}$
may be rewritten into
\[
\frac{\mathcal{W}_{n}\left(  \overline{H}\left(  h\right)  \xi_{n}%
+\overline{H}\left(  h\right)  \right)  -\mathcal{W}_{n}\left(  \overline
{H}\left(  h\right)  \right)  }{\sqrt{\overline{H}\left(  h\right)  }},
\]
where $\xi_{n}:=\overline{H}\left(  Z_{n-k:n}\right)  /\overline{H}\left(
h\right)  -1.$ Since $\overline{H}$ is regularly varying, then we may show
readily that, in the set $A_{n}\left(  \epsilon\right)  ,$ we have $\left\vert
\xi_{n}\right\vert \leq\epsilon$ too, therefore%
\[
\mathbb{P}\left(  \left\vert \vartheta_{n}\right\vert >\eta\right)  \leq
I_{n}+\mathbb{P}\left\{  A_{n}^{c}\left(  \epsilon\right)  \right\}
+\mathbb{P}\left\{  A_{n}^{c}\left(  \epsilon\right)  \right\}  ,
\]
where%
\[
I_{n}:=\mathbb{P}\left(  \sup_{0\leq t\leq\overline{H}\left(  h\right)
\xi_{n}}\left\vert \mathcal{W}_{n}\left(  t+\overline{H}\left(  h\right)
\right)  -\mathcal{W}_{n}\left(  \overline{H}\left(  h\right)  \right)
\right\vert >\eta\sqrt{\overline{H}\left(  h\right)  },\text{ }A_{n}\left(
\epsilon\right)  \right)  .
\]
Note that, for a fixed $0\leq s\leq1,$ we have
\[
\left\{  \mathcal{W}_{n}\left(  t+s\right)  -\mathcal{W}_{n}\left(  s\right)
;\text{ }0\leq t\leq1-s\right\}  \overset{d}{=}\left\{  \mathcal{W}_{n}\left(
t\right)  ;\text{ }0\leq t\leq1-s\right\}  ,
\]
it follows that $I_{n}=\mathbb{P}\left(  \sup_{0\leq t\leq\epsilon\overline
{H}\left(  h\right)  }\left\vert \mathcal{W}_{n}\left(  t\right)  \right\vert
>\eta\sqrt{\overline{H}\left(  h\right)  }\right)  .$ Since $\mathcal{W}%
_{n}\left(  t\right)  $ is a martingale, then by applying Doob's maximal
inequalities, we write%
\[
\mathbb{P}\left(  \sup_{0\leq t\leq\epsilon\overline{H}\left(  h\right)
}\left\vert \mathcal{W}_{n}\left(  t\right)  \right\vert >\eta\sqrt
{\overline{H}\left(  h\right)  }\right)  \leq\frac{\mathbf{E}\left\vert
\mathcal{W}_{n}\left(  \epsilon\overline{H}\left(  h\right)  \right)
\right\vert }{\eta\sqrt{\overline{H}\left(  h\right)  }}.
\]
Since $\mathbf{E}\left\vert \mathcal{W}_{n}\left(  \epsilon\overline{H}\left(
h\right)  \right)  \right\vert \leq\sqrt{\epsilon\overline{H}\left(  h\right)
}$ and $\mathbb{P}\left\{  A_{n}^{c}\left(  \epsilon\right)  \right\}
<\epsilon,$ thus $\mathbb{P}\left(  \left\vert \vartheta_{n}\right\vert
>\eta\right)  \leq\eta^{-1}\epsilon^{1/2}+\epsilon$ which tends to zero as
$\epsilon\downarrow0,$ as sought.\medskip
\end{proof}

\end{document}